\documentclass[12pt,reqno]{amsart}
\usepackage[T1]{fontenc}
\usepackage{lmodern}
\usepackage[margin=1.2in]{geometry}
\usepackage{mathrsfs}
\usepackage[hidelinks]{hyperref}
\hypersetup{pdfstartview={FitH}}

\newtheorem{theorem}{Theorem}
\newtheorem{lemma}{Lemma}[section]
\newtheorem{proposition}[lemma]{Proposition}
\theoremstyle{remark}
\newtheorem{remark}{Remark}
\theoremstyle{definition}
\newtheorem{problem}{Problem}

\numberwithin{equation}{section}

\title[A 21-coloring of the plane]{A 21-Coloring of the Plane Without Monochromatic Unit-Area Rectangles}
\author[Gary Hu]{Gary Hu}
\address{Williams College}
\email{gh7@williams.edu}
\author[Yumo Liu]{Yumo Liu}
\address{Bard College}
\email{yy2323@bard.edu}
\date{}


\begin{document}

\begin{abstract}
Erd\H{o}s and Graham asked whether every finite coloring of the plane must contain a monochromatic rectangle of any prescribed area. Kova\v{c} gave a negative answer by constructing a $25$-coloring of the plane with no monochromatic rectangle of area $1$. We reduce the number of colors to $21$ by replacing the square cells in his construction with regular hexagons.
\end{abstract}

\maketitle

\section{Introduction}

Euclidean Ramsey theory asks which finite point configurations must occur monochromatically in every finite coloring of Euclidean space. In 1980, Graham proved that every finite coloring of $\mathbb R^n$ contains a monochromatic $n$-simplex with any prescribed positive volume \cite{graham}. Erd\H{o}s and Graham asked whether an analogous statement holds for rectangles and parallelograms \cite{erdos-graham}. The rectangle problem was still open in the 2015 second edition of Graham and Butler's \emph{Rudiments of Ramsey Theory} \cite{graham-butler}.

In 2026, Kova\v{c} answered the rectangle question by constructing a $25$-coloring of $\mathbb R^2$ with no monochromatic rectangle of area $1$ \cite{kovac}. For rectangles, scaling shows that it suffices to consider only rectangles of area 1, so this answers the question negatively. His coloring has a stronger property: it contains no monochromatic parallelogram whose two adjacent side lengths have product $1$. We improve the number of colors from $25$ to $21$.

\begin{theorem}\label{thm:main}
  There exists a 21-coloring of $\mathbb{R}^2$ that contains no monochromatic parallelogram whose adjacent side lengths have product 1.
\end{theorem} 

Thus, the theorem gives a 21-coloring of $\mathbb R^2$ with no monochromatic rectangles of area 1. The parallelogram question remains open because area 1 does not imply that the product of adjacent side lengths is 1.

Our construction follows the approach of Kova\v{c}. We identify the plane as the complex plane $\mathbb{C}$ and define the coloring based on the position of $z^2$, replacing his periodic square tiling with a Voronoi tiling derived from a triangular lattice. The index of the sublattice involved is 21. Section~2 presents the coloring and proves Theorem \ref{thm:main}, while Section 3 demonstrates that no smaller index can satisfy the requirements of this construction.

\section{Proof of Theorem~\ref{thm:main}}

Identify $\mathbb R^2$ with $\mathbb C$. Let $P=ABCD$ be a (possibly degenerate) parallelogram. Denote its vertices by $z_A=z$, $z_B=z+u$, $z_C=z+u+v$, and $z_D=z+v$.

Consider the alternating sum of squares:
\begin{equation}\label{eq:invariant}
\mathscr I(P)=z_A^2-z_B^2+z_C^2-z_D^2.
\end{equation}
Substituting the coordinates of the vertices, we obtain $\mathscr I(P)=z^2-(z+u)^2+(z+u+v)^2-(z+v)^2=2uv$. Consequently, if the product of the lengths of two adjacent sides is $1$, then $|\mathscr I(P)|=2$. We will construct a coloring such that for any monochromatic parallelogram $P$, the corresponding value $\mathscr I(P)$ does not lie on the circle $|w|=2$.

Let $\omega=e^{\pi i/3}$, and define $L=\mathbb Z+ \mathbb Z\omega$. In the lattice $L$, the Voronoi cell $H_0$ associated with the origin $0$ is a regular hexagon with an inradius of $1/2$ and a circumradius of $1/\sqrt3$. By assigning each boundary point of the Voronoi tiling to one of its adjacent cells, we obtain a half-open version $V$ of $H_0$. The collection of translated sets $\{\lambda+V : \lambda\in L\}$ then forms a partition of $\mathbb C$. Let $\alpha=4+\omega$. 

For $j\in\{0,1,\ldots,20\}$, define
\[
\mathscr C_j=\bigl\{z\in\mathbb C:z^2\in \frac{5}{6}(p+q\omega+V)\text{, where }p,q\in\mathbb Z\text{ and }p-4q\equiv j\pmod{21}\bigr\}.
\]
Since the cells $\frac{5}{6}(\lambda+V)$ form a partition of $\mathbb{C}$, the sets $\mathscr{C}_0, \ldots, \mathscr{C}_{20}$ constitute a 21-coloring of $\mathbb{C}$.

The congruence condition is equivalent to membership in $\alpha L$: $p+q\omega\in\alpha L$ if and only if $p-4q\equiv0\pmod{21}$. If $p+q\omega=\alpha(r+s\omega)$, then $p=4r-s$ and $q=r+5s$, so $p-4q=-21s$. Conversely, if $p-4q=21k$, then $p+q\omega=\alpha((q+5k)-k\omega)$.

\begin{lemma}\label{lem:separation}
  The set $\frac{5}{6}(\alpha L+4H_0)$ is disjoint from the circle $|w|=2$.
\end{lemma}

\begin{proof}
  The circumradius of the hexagon $\frac{5}{6}(4H_0)$ is $10/(3\sqrt{3}) < 2$; therefore, this hexagon, centered at $0$, lies within the open disk $|w|<2$.
  
  Let $0 \neq h = p+q\omega \in L$. Since $|h|^2 = p^2+pq+q^2$ and $p^2+pq+q^2 \equiv (p-q)^2 \pmod{3}$, the positive integer $|h|^2$ cannot be equal to $2$. Consequently, either $|h|=1$ or $|h| \ge \sqrt{3}$.

  Since $\alpha=4+\omega$, we have $|\alpha|^2=21$. Hence, if $|h|\ge\sqrt{3}$, then $|\alpha h|=|\alpha||h|\ge\sqrt{63}$, meanwhile the circumradius of $4H_0$ is $4/\sqrt{3}$. Therefore, $d(\alpha h, 4H_0) \ge \sqrt{63} - 4/\sqrt{3} > 5/2$. If $|h|=1$, then multiplication by $h$ corresponds to a rotation that leaves $L$ and $H_0$ invariant.

  Thus, it suffices to consider the case $h=1$. In this case, $\alpha=9/2+(\sqrt{3}/2)i$, and the edge of $4H_0$ lying on the line $\operatorname{Re}z=2$ has endpoints $2\pm2i/\sqrt{3}$. The orthogonal projection of $\alpha$ onto this edge falls within the line segment, so $d(\alpha, 4H_0)=5/2$.
  
  We have shown that for any non-zero element $h \in L$, $d(\alpha h, 4H_0) \ge 5/2$. After scaling by a factor of $5/6$, we obtain $d(0, \frac{5}{6}(\alpha h+4H_0)) \ge 25/12 > 2$. Thus, every translated set (for $h \neq 0$) lies outside the disk $|w| \le 2$, whereas the set centered at $0$ lies within this disk. It follows that none of the sets $\frac{5}{6}(\alpha h+4H_0)$ can intersect the circle $|w|=2$. 
\end{proof}

\begin{proof}[Proof of Theorem~\ref{thm:main}]
  Suppose $P=ABCD$ is monochromatic, with all four of its vertices belonging to $\mathscr C_j$. For each $X\in\{A,B,C,D\}$, let $z_X^2=\frac{5}{6}(\lambda_X+\varepsilon_X)$, where $\lambda_X=p_X+q_X\omega\in L$ satisfies $p_X-4q_X\equiv j\pmod{21}$, and $\varepsilon_X\in V\subset H_0$. The alternating sum $\lambda_A-\lambda_B+\lambda_C-\lambda_D$ is congruent to $j-j+j-j=0$ modulo $21$, and thus it belongs to $\alpha L$. Since $H_0$ is centrally symmetric and convex, we have $\varepsilon_A-\varepsilon_B+\varepsilon_C-\varepsilon_D\in4H_0$.
  It follows that $\mathscr I(P)\in\frac{5}{6}(\alpha L+4H_0)$. Lemma~\ref{lem:separation} shows that this set is disjoint from the circle $|w|=2$. Therefore, $|\mathscr I(P)|\ne2$, and the product of the lengths of any two adjacent sides of $P$ is not equal to $1$.
\end{proof}

\section{Sharpness of the lattice construction}

Next, we demonstrate that it is impossible to achieve a smaller index with this construction. Let $M \subset L$ be a sublattice of finite index. Suppose that, after scaling, a copy of $4H_0$ centered at $0$ lies within the region $|w|<2$, while copies centered at non-zero points of $M$ lie outside the closed disk $|w|\le2$. The first condition requires the scaling factor to be less than $\sqrt{3}/2$. Under this assumption, a necessary condition for the second requirement to hold is that for any non-zero point $m \in M$, the distance $d(m, 4H_0)$ must be greater than $4/\sqrt{3}$.

\begin{proposition}\label{prop:sharp}
  If the index of $M \subset L$ is at most $20$, then there must exist a non-zero point $m \in M$ such that $d(m, 4H_0) \le 4/\sqrt{3}$.
\end{proposition}

\begin{proof}
  In lattice coordinates, a point $x+y\omega$ belongs to $4H_0$ if and only if the following conditions are satisfied: $|2x+y|\le4$, $|x+2y|\le4$, and $|y-x|\le4$.
 Let $Q=4H_0-(1+\omega)/3$,
 where the translation is chosen so that $Q$ contains exactly $21$ lattice points.
  %Let $Q=4H_0-(1+\omega)/3$. 
  Then the corresponding conditions then become $|2p+q+1|\le4$, $|p+2q+1|\le4$, and $|q-p|\le4$. The lattice points in $Q$ are as follows:

\begin{center}
\begin{tabular}{c|l}
$q$ & $p$ \\ \hline
$-3$ & $1$ \\
$-2$ & $-1,0,1,2$ \\
$-1$ & $-2,-1,0,1,2$ \\
$0$ & $-2,-1,0,1$ \\
$1$ & $-3,-2,-1,0,1$ \\
$2$ & $-2,-1$
\end{tabular}
\end{center}

There are $21$ lattice points in $Q$. If $[L:M] \le 20$, then there must exist two lattice points (let us call them $x$ and $y$) that belong to the same coset of $M$. Their difference, $m = x - y$, is a non-zero point in $M$. Since $Q$ is a translate of $4H_0$, it follows that $m \in Q - Q = 8H_0$. Consequently, $m/2 \in 4H_0$, and $|m| \le 8/\sqrt{3}$. This implies that $d(m, 4H_0) \le |m - m/2| = |m|/2 \le 4/\sqrt{3}$.
\end{proof}

Multiplying by $\alpha = 4 + \omega$ scales the Euclidean area by a factor of $|\alpha|^2 = 21$, so $[L : \alpha L] = 21$. Lemma~\ref{lem:separation} shows that an index of $21$ is feasible, while Proposition~\ref{prop:sharp} rules out all indices up to and including $20$. Therefore, $21$ is the smallest index for which this sub-lattice construction ensures that the translated sets do not intersect the circle $|w| = 2$.

\begin{remark}
  Recall that a coloring is said to be Jordan-measurable if the boundary of each color class has Lebesgue measure zero. The 25-coloring proposed by Kova\v{c} is Jordan-measurable, as is the coloring constructed in this paper. In our coloring, the boundary of each color class is contained within the preimage (under the mapping $z \mapsto z^2$) of the union of the boundaries of the Voronoi cells. This union consists of a countable collection of line segments, and the preimage of each such segment is contained within a conic that has planar measure zero.
\end{remark}

\begin{problem}\label{prob:min-colors}
  Determine the smallest integer $r$ such that there exists an $r$-coloring of $\mathbb{R}^2$ containing no monochromatic rectangle of area 1.
\end{problem}

Theorem~\ref{thm:main} establishes that $r \le 21$. Proposition~\ref{prop:sharp} demonstrates that the construction method presented in Section 2 cannot achieve a smaller number of colors.

\section*{Acknowledgments}

This work was completed as part of the REU on Combinatorics, AI, and Algorithms for Real Problems (REU-CAAR), held at the University of Maryland and supported by the National Science Foundation under Grant No. 2447466. We thank William Gasarch for introducing us to this problem and for his guidance throughout this project.

\end{document}